\documentclass[11pt]{amsart}

\usepackage{amsmath,amsfonts,amsthm,amsopn,amssymb}
\usepackage{verbatim,wasysym,mathrsfs}
\usepackage[left=2.6cm,right=2.6cm,top=3.1cm,bottom=3.1cm]{geometry}
\usepackage{cite,marginnote}
\usepackage{color,enumitem,graphicx}
\usepackage[colorlinks=true,urlcolor=blue, citecolor=red,linkcolor=blue,
linktocpage,pdfpagelabels, bookmarksnumbered,bookmarksopen]{hyperref}
\usepackage[english]{babel}
\usepackage[T1]{fontenc}
\usepackage{lmodern}
\usepackage[hyperpageref]{backref}

\numberwithin{equation}{section}

\makeindex
\newtheorem{theorem}{Theorem}[section]
\newtheorem{proposition}[theorem]{Proposition}
\newtheorem{lemma}[theorem]{Lemma}
\newtheorem{remark}[theorem]{Remark}
\newtheorem{example}[theorem]{Example}
\newtheorem{corollary}[theorem]{Corollary}
\newtheorem{definition}[theorem]{Definition}

\newcommand{\R}{\mathbb R}
\newcommand{\al}{\alpha}

\newcommand{\e}{\varepsilon}

\newcommand{\bt}{\begin{theorem}}
\newcommand{\et}{\end{theorem}}
\newcommand{\bl}{\begin{lemma}}
\newcommand{\el}{\end{lemma}}
\newcommand{\bd}{\begin{definition}}
\newcommand{\ed}{\end{definition}}
\newcommand{\bc}{\begin{corollary}}
\newcommand{\ec}{\end{corollary}}
\newcommand{\bp}{\begin{proof}}
\newcommand{\ep}{\end{proof}}
\newcommand{\bx}{\begin{example}}
\newcommand{\ex}{\end{example}}
\newcommand{\bi}{\begin{exercise}}
\newcommand{\ei}{\end{exercise}}
\newcommand{\bo}{\begin{proposition}}
\newcommand{\eo}{\end{proposition}}
\newcommand{\br}{\begin{remark}}
\newcommand{\er}{\end{remark}}
\newcommand{\be}{\begin{equation}}
\newcommand{\ee}{\end{equation}}
\newcommand{\ba}{\begin{align}}
\newcommand{\ea}{\end{align}}
\newcommand{\bn}{\begin{enumerate}}
\newcommand{\en}{\end{enumerate}}
\newcommand{\bg}{\begin{align*}}
\newcommand{\bcs}{\begin{cases}}
\newcommand{\ecs}{\end{cases}}

\def\R{\mathbb R}

\def\N{\mathbb N}

\def\Proof{\noindent{\bf Proof}\quad}
\def\qed{\hfill$\square$\smallskip}
\makeatletter
\def\@makefnmark{}
\makeatother

\newcommand{\bean}{\begin{eqnarray*}}
\newcommand{\eean}{\end{eqnarray*}}

\renewcommand{\triangle}{\Delta}
\renewcommand{\epsilon}{\varepsilon}

\title[Positive solutions to the planar logarithmic Choquard equation]
{Positive solutions to the planar logarithmic Choquard equation via asymptotic approximation}

\author[D.~Cassani]{Daniele Cassani}
\author[L.~Du]{Lele Du}
\address[D.~Cassani and L.~Du]{\newline\indent Dipartimento di Scienza e Alta Tecnologia
	\newline\indent
	Universit\`{a} degli Studi dell'Insubria
	\newline\indent and
	\newline\indent RISM--Riemann International School of Mathematics
	\newline\indent Villa Toeplitz, Via G.B. Vico, 46 -- 21100 Varese, Italy.}
    \email{\href{mailto:daniele.cassani@uninsubria.it}{daniele.cassani@uninsubria.it}}
     \email{\href{mailto:dldu@uninsubria.it}{dldu@uninsubria.it}}
\author[Z.~Liu]{Zhisu Liu}
\address[Z.~Liu]{\newline\indent Center for Mathematical Sciences/School of Mathematics and Physics
\newline\indent China University of Geosciences
\newline\indent Wuhan, Hubei, 430074, PRC
\newline\indent and
\newline\indent Dipartimento di Scienza e Alta Tecnologia
	\newline\indent
	Universit\`{a} degli Studi dell'Insubria
	\newline\indent Villa Toeplitz, Via G.B. Vico, 46 -- 21100 Varese, Italy.}
    \email{\href{mailto:liuzhisu@cug.edu.cn}{liuzhisu@cug.edu.cn}}

\thanks{Corresponding author:
 D.~Cassani ({\tt daniele.cassani@uninsubria.it}).}

\thanks{Z. Liu is supported by the National Natural Science Foundation
of China (No.~12226331), and the Fundamental Research Funds for the Central
Universities, China University of Geosciences (Wuhan, No. CUG2106211; CUGST2).
}

\date{\today}

\begin{document}

\begin{abstract}
In this paper we study the following nonlinear Choquard equation
$$
-\triangle u+u=\left(\ln\frac{1}{|x|}\ast F(u)\right)f(u),\quad\text{ in }\,\R^2,
$$
where $f\in C^1(\R,\R)$ and $F$ is the primitive of the nonlinearity $f$ vanishing at zero. We use an asymptotic approximation approach to establish the existence of positive solutions to the above problem in the standard Sobolev space $H^1(\R^2)$. We give a new proof and at the same time extend part of the results established in [Cassani--Tarsi, Calc.~Var.~P.D.E. (2021)].
\end{abstract}

\maketitle


\section{Introduction}
\label{sec1}
\noindent Let us consider the following class of nonlocal equations
\begin{equation}\label{eqn:log-choquard0}
-\triangle u+u=\left(\ln\frac{1}{|x|}\ast F(u)\right)f(u),\quad\text{ in }\,\R^2,
\end{equation}
where $F\in C^1(\R,\R)$ is the primitive function of the nonlinearity $f$ vanishing at zero.
This two-dimensional problem has remained open for a long time because of the sign-changing nature of the Coulomb interaction, given by the convolution,
for which variational methods do not straightforward apply.

\noindent Indeed, on the one hand, equation \eqref{eqn:log-choquard0} has a variational structure, in the sense that at least formally, it  turns out to be the Euler-Lagrange equation related to the energy functional
\begin{equation}\label{eqn:log-fun}
I(u)=\frac{1}{2}\int_{\R^2}\bigg(|\nabla u|^2+u^2\bigg){\rm d}x+\frac{1}{2}\int_{\R^2}\int_{\R^2}\ln(|x-y|)F(u(y))F(u(x)){\rm d}x{\rm d}y.
\end{equation}
On the other hand, the second term in \eqref{eqn:log-fun} is not well-defined on the natural Sobolev space $H^1(\R^2)$. A few attempts to overcome this difficulty have been done, in particular for power-like nonlinearities, see \cite{CW} and references therein, where the finiteness of such logarithmic convolution term is required as additional condition, which settles the problem in a proper subspace of $H^1(\R^2)$. However, it is well known how the smaller the space the more difficult is to get energy estimates in order to prove existence results. Recently, in \cite{Cassani21} a different approach has been developed by means of a new wighted version of the Trudinger--Moser inequality, for which the problem is well defined in a log-weighted Sobolev space where variational methods can be applied up to cover the maximal possible nonlinear growth, which in dimension two it is of exponential type. See also \cite{Bucur2022} for further extensions in this direction.

\medskip

\noindent Let us make the following assumptions on the nonlinearity $f$:

\medskip 

\begin{itemize}
\item[ ($f_1$)] $f(s)\geq0$ for all $s\geq0$ and $f(s)\equiv0$ for all $s\leq0$. $f(s)\leq Cs^pe^{4\pi s^2}$ as $s\rightarrow+\infty$ for some $p>0$ and $f(t)=o(t)$ as $t\rightarrow0^+$.
\item[ ($f_2$)] there exist $C>1>\tau>0$ such that $\tau\leq \frac{F(s)f'(s)}{f^2(s)}\leq C$ for any $s>0$;
\item[ ($f_3$)] $\lim\limits_{s\rightarrow+\infty} \frac{F(s)f'(s)}{f^2(s)}=1$ or equivalently
$\lim\limits_{s\rightarrow+\infty}\frac{d}{ds}\frac{F(s)}{f(s)}=0$;
\item[ ($f_4$)] there exist $\beta>0$ and $0<\rho<1/4$ such that
$$
\lim\limits_{t\rightarrow+\infty}\frac{tF(t)}{e^{4\pi t^{2}}}\geq \beta>\frac{1}{\rho^2\sqrt{\ln2}\cdot\pi^\frac{3}{2}}.
$$
\end{itemize}
Notice that from ($f_3$), we deduce that there exist $\tau_0\in(0,1)$ and
$t_0>0$ such that $\frac{F(t)f'(t)}{f^2(t)}\geq \tau_0$ for any $t\geq t_0$, which in turn implies that
$$
\int_{t_0}^t\frac{f'(s)}{f(s)}{\rm d}s\geq \tau_0\int_{t_0}^t\frac{f(s)}{F(s)}{\rm d}s.
$$
So, from ($f_1$) we have that there exist $M_0>0$ and
$t_0>0$ such that
\begin{equation}\label{eqn:f5}
F(t)\leq M_0|f(t)|,\quad  |t|\geq t_0.
\end{equation}
\noindent Condition $(f_1)$ is the usual exponential growth assumption as prescribed by the Trudinger-Moser inequality, which gives the following
\begin{equation}\label{eqn:f1}
    0\leq F(t)\leq C\cdot\left\{\begin{array}{ll}
    t^2,&  t\leq \bar{t},\\
    t^{p-1}e^{4\pi t^2}, & t\geq \tilde{t}
  \end{array}
  \right.
\end{equation}
for some $\tilde{t}>\bar{t}>0$. Assumptions ($f_2$) and ($f_3$) have been introduced in \cite{Cassani21} and turn out to be the key ingredients in order to prove boundedness of Palais-Smale sequences. In fact, the difficulty here to obtain the boundedness of Palais-Smale sequences is due to the presence of a non-homogenous nonlinearity $f$ for which the Ambrosetti-Rabinowitz condition fails to work. Assumption ($f_4$)  is in the spirit of de Figueiredo-Miyagaki-Ruf condition \cite{Figueiredo95} which in dimension two turns out to be the key ingredient in order to have compactness. (It is somehow the equivalent of the Brezis-Nirenberg condition for the mountain pass level $c<\frac{1}{N}S^{N/2}$ in higher dimensions $N\geq 3$, where $S$ is the optimal constant for the critical Sobolev embedding $H^1\hookrightarrow L^{2^*}$.)

\medskip
\noindent An explicit example of function $F(t)$ satisfying our set of assumptions is the following:
\begin{equation}\label{eqn:f1+}
  F(t)=\frac{t^2}{\ln\frac{1}{\ln(1+t)}}e^{4\pi t^2},\quad t\in(0,+\infty).
\end{equation}
Let us point out that the result of \cite{Cassani21} does not cover this kind of nonlinearities and in this respect, the approach developed here is an improvement which allows for more general nonlinearities.

\medskip

\noindent Our main result is the following:
\begin{theorem}\label{Thm:Existence}
Suppose the nonlinearity $f$ satisfies $(f_1)$--$(f_4)$. Then, problem (\ref{eqn:log-choquard0}) possesses a positive solution $u\in H^{1}(\R^2)$.
\end{theorem}
\noindent We observe that as a consequence of \cite{Cingolani22}, the solution obtained in Theorem \ref{Thm:Existence} is actually radially symmetric, up to translations, and strictly decreasing.

\medskip

\noindent From the point of view of applications, equation \eqref{eqn:log-choquard0} boasts a longstanding presence in quite different nonlinear contexts, from the early studies on Polarons by Fr\"ohlich in the '30s, to more recent applications to quantum gravity and plasma physics; see \cite{MV} and references therein. The mathematical success, iniziated by Lieb in the '70s \cite{Lieb} (let us mention that in the last two years more than 200 papers have been devoted to this topic) is due to the richness of the framework with challenges which range form Functional Analysis, to local and nonlocal systems of PDEs and Potential Theory, see \cite{Bucur2022, Cingolani22}. Let us emphasize that a major difficulty in studying the so-called limiting case \eqref{eqn:log-choquard0}, is the lack of a proper function space setting.  The approach developed in  \cite{Cassani21} in dimension two and the extended to any dimension in \cite{Bucur2022}, consists of introducing a logarithmic weight on the mass part of the Sobolev norm, namely  
\begin{equation}\label{wn}
\|u\|^2:=\|u\|_{H^1}^2+\bigg(\int_{\R^2}|u|^q\log(1+|x|){\rm d}x\bigg)^{\frac{2}{q}},\quad q>2,
\end{equation}
and completing smooth compactly supported functions with respect to this norm. This, together with a suitable version of the Trudinger-Moser inequatlity yields a weighted Sobolev space in which the energy functional is well-defined and of class $C^1$. This opens the way to variational methods which provide mountain pass type solutions to \eqref{eqn:log-choquard0} which have finite energy in terms of \eqref{wn}. However, observe that the norm \eqref{wn} is not invariant under translations whereas the
energy functional $I$ does. Moreover, the quadratic part of the energy functional is never coercive in that context.  
Because of the above reasons, the authors in \cite{Cassani21,Bucur2022} can consider nonlinearities $f$ whose asymptotic behavior near the origin is given by $s^q$, with $q>1$, which plays an crucial role in proving the boundedness of PS sequences. A major purpose of this paper is to remove this restriction and allow to consider the wider range of super-linear nonlinearities, that is $f(t)=o(t)$, as $t\rightarrow0$.

\medskip

\noindent Here we exploit an asymptotic approximation procedure developed in \cite{LRTZ2021} to set the problem (\ref{eqn:log-choquard0}) in the standard Sobolev space $H^1(\R^2)$. (See also \cite{BV} for further penalization methods in the case of power-like nonlinearities.)
 In order to overcome the difficulty of the sign-changing Newtonian kernel, we modify equation (\ref{eqn:log-choquard0}) as follows
\begin{equation}\label{eqn:2-spos4}
    -\triangle {u}+ u-\bigg[G_{\alpha}(x)\ast F(u)\bigg]f(u)=0, \quad \mbox{ in }\,\, \R^2,
\end{equation}
where $\alpha\in(0,1)$ and
$$
\lim\limits_{\alpha\rightarrow0^+}G_{\alpha}(x):=\lim\limits_{\alpha\rightarrow0^+}\frac{|x|^{-\alpha}-1}{\alpha}=-\ln|x|
$$
for $x\in\R^2\setminus\{0\}$.
The corresponding energy functional to (\ref{eqn:2-spos4})
is well defined in $H^1(\R^2)$ for fixed $\alpha\in(0,1)$,
which enables us to use minimax methods to establish the existence of
positive solutions for (\ref{eqn:2-spos4}).

\noindent By passing to the limit as $\alpha\rightarrow0^+$, a convergence argument within $H^1(\R^2)$ allows us to prove that the limit solution turns out to be a positive solutions to the original problem (\ref{eqn:log-choquard0}). Here the main difficulty is the balance between the too loose sign-changing logarithm kernel and the exponential growth rate of a fairly general nonlinearity. Some uniform bounds with respect to the asymptotic approximation parameter $\al$ turn out to be the key in order to get compactness and then existence of a finite energy (in the $H^1$ sense) solution. We are confident that the methods developed here might be useful in other contexts.

\section{Preliminaries}\label{sec2}
\noindent For $1\leq s\leq +\infty$, we denote by $\|\cdot\|_s$ the usual
norm of the Lebesgue space $L^s(\R^2)$ as well as
$
H^1(\R^2):=\{u\in L^2(\R^2):\nabla u\in L^2(\R^2)\}
$
is the usual Sobolev space endowed with the norm
$\|u\|:=\left(\int_{\R^2}(|\nabla u|^2+u^2){\rm d}x\right)^{\frac{1}{2}}$. In the sequel, when it is not misleading, constants $C$ may change value from line to line. 

\noindent Let us next recall some basic facts starting with the Hardy-Littlewood-Sobolev inequality, see for instance \cite{Bennett88}, which will be frequently used throughout this paper.
\begin{lemma}(Hardy-Littlewood-Sobolev inequality)\label{Lem:HLS}
Let $s,r>1$ and $\alpha\in(0,N)$ with $1/s+\alpha/N+1/r=2$, $f\in L^s(\R^N)$ and $h\in L^r(\R^N)$. There exists a sharp constant
$C_{s,N,\alpha,r}$ independent of $f,h$ such that
$$
\int_{\R^N}\bigg[\frac{1}{|x|^\alpha}\ast f(x)\bigg]h(x){\rm d}x\leq C_{s,N,\alpha,r}\|f\|_s\|h\|_r.
$$
If $t=s=\frac{2N}{2N-\alpha}$, then
$$
C_{s,N,\alpha,r}=C_{N,\alpha}=\pi^{\alpha/2}\frac{\Gamma(\frac{N}{2}-
\frac{\alpha}{2})}{\Gamma(N-\frac{\alpha}{2})}\bigg\{\frac{\Gamma(\frac{N}{2})}{\Gamma(N)}\bigg\}^{-1+\frac{\alpha}{N}},
$$
and if $N=2,\alpha\in(0,1]$, then $C_{2,\alpha}\leq 2\sqrt{\pi}$.
\end{lemma}

\noindent As afore mentioned, in dimension two the maximal degree of summability for functions with membership in $H^1(\mathbb{R}^2)$ is of exponential type, for which we recall the first result available in the whole plane due to \cite{Cao92}:
\begin{equation}\label{Lem:MTI}
\sup_{\|\nabla u\|_2\leq 1,\, \|u\|_2\leq M}\int_{\R^2}\left(e^{\alpha|u|^2}-1\right)\,{\rm d}x\leq C(\alpha)\|u\|_2<\infty, \text{ if } \alpha<4\pi.
\end{equation}
(See also \cite{CST} for a wide context on this topic.)
\begin{lemma}\label{Lem:jianjin}
Assume that $u_n\rightarrow u$ in $H^1(\R^2)$,
then there exists $C>0$ independent of $n$ such that
$$
\int_{\R^2}\frac{F(u_n(y))}{|x-y|^{\nu}}{\rm d}y\leq C
$$
uniformly for $x\in\R^2$, where $\nu\in(0,1)$.
\end{lemma}
\Proof
By H\"{o}lder's inequality, ($f_1$) and ($f_2$), one has
$$
\int_{\R^2}F(u_n(y)){\rm d}y<+\infty
$$
uniformly for $n$, since $u_n\rightarrow u$ in $H^1(\R^2)$ (see \cite{do14}). It then follows from ($f_1$) and ($f_2$) that
\begin{equation}\label{eqn:2-F}
\aligned
\int_{\R^2}\frac{F(u_n(y))}{|x-y|^{\nu}}{\rm d}y
&=\int_{|x-y|\leq1}\frac{F(u_n(y))}{|x-y|^{\nu}}{\rm d}y+\int_{|x-y|\geq1}\frac{F(u_n(y))}{|x-y|^{\nu}}{\rm d}y\\
&\leq C+\bigg(\int_{|x-y|\leq1}\frac{1}{|x-y|^{2\nu}}{\rm d}y\bigg)^{1/2}\cdot\bigg(\int_{\R^2}F^2(u_n(y)){\rm d}y\bigg)^{1/2}\\
&\leq C+C\bigg(\int_{\R^2}|u_n|^4+(e^{9\pi |u_n|^2}-1){\rm d}y\bigg)^{1/2}\\
&\leq C,
\endaligned
\end{equation}
where the last inequality follows from the fact that $u_n\rightarrow u$ in $H^1(\R^2)$.
\qed
%
%

\noindent The proof of the following Lemma is standard and we omit it.
\begin{lemma}\label{Lem:G}
For any $\alpha\in(0,1]$, there exists $C_\beta>0$ such that
$$
\frac{s^{-\alpha}-1}{\alpha}\leq C_\beta s^{-\beta}
$$
holds for all $\beta\in(\alpha,+\infty)$ and $s>0$.
\end{lemma}

\section{The asymptotic approximation method}
\subsection{The modified equation}\setcounter{equation}{0}
\noindent Recently in \cite{Cingolani22}, the authors used the method of moving plans to prove that
positive solutions of (\ref{eqn:log-choquard0}) are radially symmetric. Based on this fact, it is natural to restrict ourself to the space $H_r^1(\R^2)$ of radially symmetric functions.

\medskip

\noindent Set $G_\alpha(x)=\frac{|x|^{-\alpha}-1}{\alpha}$, $\alpha\in(0,1)$, $x\in\R^2$ and consider the following equation
\begin{equation}\label{eqn:2-spos}
    -\triangle {u}+ u=\bigg[G_{\alpha}(x)\ast F(u)\bigg]f(u)
\end{equation}
which has the associated energy functional
$$
I_{\alpha}(u)=\frac{1}{2}\int_{\R^2}\bigg(|\nabla u|^2+u^2\bigg){\rm d}x+\frac{1}{2\alpha}
\bigg[\int_{\R^2}F(u){\rm d}x\bigg]^2-\frac{1}{2\alpha}\int_{\R^2}\bigg[\frac{1}{|x|^{\alpha}}*F(u)\bigg]F(u){\rm d}x.
$$

\subsection{Regularity of the modified energy functional}

According to the definition of $G_{\alpha}$, by means of the Hardy-Littleword-Sobolev inequality, let us prove the following 


\begin{lemma}\label{Lem:functional2}
Let $\al>0$, then $I_{\alpha}\in C^1(H_r^1(\R^2),\R)$ and
$$
 I'_{\alpha}(u)v= \int_{\R^2}\bigg(\nabla u\nabla v+uv\bigg){\rm d}x+\frac{1}{\alpha}\int_{\R^2}F(u){\rm d}x\cdot\int_{\R^2}f(u)v{\rm d}x
 -\frac{1}{\alpha}\int_{\R^2}\bigg[\frac{1}{|x|^{\alpha}}*F(u)\bigg]f(u)v{\rm d}x
$$
for any $u,v\in H_r^1(\R^2)$.
\end{lemma}
\Proof
It is standard to check that for fixed $\al>0$, one has $I_\al\in C(H_r^1(\R^2),\R)$. By straightforward calculations, the G\^{a}teaux derivative of $I_\al$ at $u\in H_r^1(\R^2)$ is given by
$$
 I'_{\alpha}(u)v= \int_{\R^2}\bigg(\nabla u\nabla v+uv\bigg){\rm d}x+\frac{1}{\alpha}\int_{\R^2}F(u){\rm d}x\cdot\int_{\R^2}f(u)v{\rm d}x
 -\frac{1}{\alpha}\int_{\R^2}\bigg[\frac{1}{|x|^{\alpha}}*F(u)\bigg]f(u)v{\rm d}x,
$$
for any $v\in H_r^1(\R^2)$.
It remains to prove that $ I'_{\alpha}(u_n)\rightarrow I'_{\alpha}(u)$ if $u_n\rightarrow u$ in $H_r^1(\R^2)$. Now on the one hand,
observe that uniformly in $v\in H_r^1(\R^2)$, with $\|v\|\leq1$,
\begin{equation}\label{eqn:function1}
\aligned
&\bigg|\int_{\R^2}F(u_n){\rm d}x\cdot\int_{\R^2}f(u_n)v{\rm d}x-\int_{\R^2}F(u){\rm d}x\cdot\int_{\R^2}f(u)v{\rm d}x\bigg|\\
&\leq \int_{\R^2}F(u_n){\rm d}x\int_{\R^2}|f(u_n)v-f(u)v|{\rm d}x+\int_{\R^2}\big|F(u_n)-F(u)\big|{\rm d}x\cdot\int_{\R^2}f(u)v{\rm d}x\\
&\leq C\int_{\R^2}|f(u_n)v-f(u)v|{\rm d}x+o_n(1)\\
&\leq C\bigg(\int_{\R^2}|f(u_n)-f(u)|^2{\rm d}x\bigg)^{1/2}\cdot\|v\|+o_n(1)\|v\|.
\endaligned
\end{equation}
On the other hand, by ($f_1$) and ($f_2$), one has for any $x\in\R^2$ and any $\e>0$, there exists $C_\e>0$ such that
$$
|f(u_n)|^2\leq C_\e\bigg(|u_n|^2+e^{(8\pi+\e)|u_n|^2}\bigg)=:g(u_n).
$$
It is easy to check that $g(u_n)\rightarrow g(u)$ in $L^1(\R^2)$ due to the fact that $u_n\rightarrow u$ in $H_r^1(\R^2)$.
Hence, by the Lebesgue dominated convergence theorem, we get $\|f(u_n)\|_2^2\rightarrow\|f(u)\|_2^2$ as $n\rightarrow\infty$. Therefore, from (\ref{eqn:function1}) one has that, uniformly
in $v\in H_r^1(\R^2)$, with $\|v\|\leq1$,
\begin{equation}\label{eqn:function2}
\aligned
\bigg|\int_{\R^2}F(u_n){\rm d}x\cdot\int_{\R^2}f(u_n)v{\rm d}x-\int_{\R^2}F(u){\rm d}x\cdot\int_{\R^2}f(u)v{\rm d}x\bigg|=o_n(1).
\endaligned
\end{equation}
Moreover, by Lemma \ref{Lem:jianjin}-$i)$ and the Hardy-Littleword-Sobolev inequality,
we have for any $v\in H_r^1(\R^2)$, with $\|v\|\leq1$,
\begin{equation}\label{eqn:function3}
\aligned
&\bigg|\int_{\R^2}\bigg[\frac{1}{|x|^{\alpha}}*F(u_n)\bigg]f(u_n)v{\rm d}x
-\int_{\R^2}\bigg[\frac{1}{|x|^{\alpha}}*F(u)\bigg]f(u)v{\rm d}x\bigg|\\
&\leq \int_{\R^2}\bigg[\frac{1}{|x|^{\alpha}}*F(u_n)\bigg]|f(u_n)-f(u)|v{\rm d}x
+\int_{\R^2}\bigg[\frac{1}{|x|^{\alpha}}*|F(u_n)-F(u)|\bigg]f(u)v{\rm d}x\\
&\leq C\int_{\R^2}|f(u_n)-f(u)|v{\rm d}x+\|F(u_n)-F(u)\|_{\frac{4}{4-\al}}\|f(u)v\|_{\frac{4}{4-\al}}\\
&\leq C o_n(1)\|v\|+C\|F(u_n)-F(u)\|_{\frac{4}{4-\al}}\|v\|.
\endaligned
\end{equation}
Observe that by the mean value theorem, H\"{o}lder's inequality, ($f_1$) and ($f_2$), there exists $\theta\in(0,1)$ such that
$$
\aligned
&\|F(u_n)-F(u)\|^{\frac{4}{4-\al}}_{\frac{4}{4-\al}}\\
&\leq\int_{\R^2}|f(u_n+\theta(u_n-u))(u_n-u)|^{\frac{4}{4-\al}}{\rm d}x\\
&\leq C\int_{\R^2} |(|u_n|+|u|)(u_n-u)|^{\frac{4}{4-\al}}{\rm d}x
+ C\int_{\R^2}(|u_n|+|u|)^{\frac{4p}{4-\al}}|u_n-u|^{\frac{4}{4-\al}}e^{{\frac{16\pi}{4-\al}}(|u_n|+|u_0|)^2}{\rm d}x\\
&\leq o_n(1)+ C\bigg(\int_{\R^2}(|u_n|+|u|)^{\frac{4pt}{4-\al}}|(u_n-u)|^{\frac{4t}{4-\al}}{\rm d}x\bigg)^{\frac{1}{t}}
\bigg(\int_{\R^2}e^{{\frac{16\pi t'}{4-\al}}\pi(|u_n|+|u_0|)^2}{\rm d}x\bigg)^{\frac{1}{t'}}\\
&\leq o_n(1)+ Co_n(1)\cdot\bigg(\int_{\R^2}e^{{\frac{128\pi t'}{4-\al}}\pi(|u_n-u_0|^2+|u_0|^2)}{\rm d}x\bigg)^{\frac{1}{t'}}\\
&\leq C o_n(1),
\endaligned
$$
where $\frac{1}{t}+\frac{1}{t'}=1$. Then, from (\ref{eqn:function3}) we have
$$
\bigg|\int_{\R^2}\bigg[\frac{1}{|x|^{\alpha}}*F(u_n)\bigg]f(u_n)v{\rm d}x
-\int_{\R^2}\bigg[\frac{1}{|x|^{\alpha}}*F(u)\bigg]f(u)v{\rm d}x\bigg|=o_n(1)\|v\|,
$$
which together with (\ref{eqn:function2}) completes the proof.

\subsection{Critical points of the modified energy functional}

 \begin{lemma}\label{Lem:MP1}
Assume that $(f_1)$--$(f_4)$ hold, then we have:
\begin{itemize}
 \item[$i)$] there exist constants $\rho,\eta>0$ such that $I_\al|_{S_\rho}\geq
\eta>0$ for all
$$
u\in S_\rho=\{u\in H_r^1(\R^2):\,\,\|u\|=\rho\};
$$
\item[$ii)$] there exists $e\in H^1_r(\R^2)$ with $\|e\|>\rho$ such that
 $I_\al(e)<0$.
\end{itemize}
 \end{lemma}
\begin{proof}
We first assume that $u\in H_r^{1}(\R^2)$ and $\|u\|\leq \theta$ for some $\theta\in(0,1)$. Obviously, $\int_{\R^2}|\nabla u|^2<1$.
Moreover, by ($f_1$)-($f_2$) and (\ref{eqn:f1}), there is $C_\theta>0$ such that
$$
|F(u)|^{\frac{4}{3}}\leq C\bigg(|u|^{\frac{8}{3}}+|u|^{\frac{4}{3}(p+1)}e^{\frac{8\pi}{3}|u|^{2}}\bigg).
$$
So by ($f_1$), the Hardy-Littlewood-Sobolev inequality yields
\begin{equation}\label{eqn:guestimate}
\aligned
&\int_{\R^2}\int_{\R^2}\frac{1}{|x-y|} F(u(y))F(u(x)){\rm d}x{\rm d}y\\
&\leq C\bigg(\int_{\R^2}|F(u)|^{\frac{4}{3}}{\rm d}x\bigg)^{\frac{3}{2}}\\
&\leq C\bigg(\int_{\R^2}|u|^{\frac{8}{3}}+|u|^{\frac{4}{3}(p+1)}e^{\frac{8\pi}{3}|u|^{2}}{\rm d}x\bigg)^\frac{3}{2}\\
&\leq C\bigg(\|u\|_{\frac{8}{3}}^{4}+\|u\|_{\frac{4\varsigma}{3}(p+1)}^{2(p+1)}\bigg),\quad \varsigma>1.
\endaligned
\end{equation}
From (\ref{eqn:guestimate}) and Lemma \ref{Lem:G} we deduce that
\begin{equation}\label{eqn:function}
\aligned
I_{\alpha}(u)&=\frac{1}{2}\|u\|^{2}-\frac{1}{2}\int_{\R^2}\int_{\R^2}
\frac{|x-y|^{-\alpha}-1}{\alpha}F(u(y))F(u(x)){\rm d}x{\rm d}y\\
&\geq\frac{1}{2}\|u\|^{2}-\frac{1}{2}\int\int_{|x-y|\leq1}
\frac{|x-y|^{-\alpha}-1}{\alpha}F(u(y))F(u(x)){\rm d}x{\rm d}y\\
&\geq\frac{1}{2}\|u\|^{2}-\frac{1}{2}\int\int_{|x-y|\leq1}
\frac{F(u(y))F(u(x))}{|x-y|}{\rm d}x{\rm d}y\\
&\geq\frac{1}{2}\|u\|^{2}-C\bigg(\|u\|_{\frac{8}{3}}^{4}+\|u\|_{\frac{4\varsigma}{3}(p+1)}^{2(p+1)}\bigg).
\endaligned
\end{equation}
So, let $\|u\|=\rho>0$ be small
enough, it is easy to check that there exists $\eta>0$ such that
$I_{\alpha}(u)\geq \eta$ for any $\alpha \in (0,1)$: this completes the proof of $i)$.

\medskip

\noindent For the proof of $ii)$, let us take $e_0\in H^1_r(\R^2)$ satisfies $e_0(x)\equiv1$ for $x\in B_{\frac{1}{8}}(0)$, $e_0(x)\equiv0$ for $x\in \R^2\setminus{B_{\frac{1}{4}}(0)}$ and $|\nabla e_0(x)|\leq C$. Set
\begin{equation}\label{eqn:AR-0}
\Psi(t):=\frac{1}{2}\bigg(\int_{\R^2}F\big(te_0\big){\rm d}x\bigg)^2,
\end{equation}
then by ($f_2$) we have $F(t)\leq(1-\tau)f(t)t$ for $t\geq0, \tau\in(0,1)$, and then
$$
\frac{\Psi'(t)}{\Psi(t)}\geq\frac{2}{(1-\tau)t}\quad\text{ for\,all }\,t>0.
$$
Integrating this over $[1,s]$, we find
\begin{equation}\label{eqn:AR}
\Psi(s)=\frac{1}{2}\bigg(\int_{\R^2}F\big(se_0\big){\rm d}x\bigg)^2\geq \Psi(1)s^\frac{2}{1-\tau}.
\end{equation}
Note that
$$
\frac{s^{-\alpha}-1}{\alpha}\geq \ln\frac{1}{s},\quad \text{for}\, s\in(0,1].
$$
From (\ref{eqn:AR-0}) and (\ref{eqn:AR}) we have
\begin{equation}\label{eqn:MPG}
\aligned
I_{\alpha}(te_0)
&=\frac{t^2}{2}\|e_0\|^2
-\frac{1}{2}\int_{\R^2}\int_{\R^2}
\frac{|x-y|^{-\alpha}-1}{\alpha}F\big(te_0(y)\big)F\big(te_0(x)\big){\rm d}x{\rm d}y\\
&=\frac{t^2}{2}\|e_0\|^2
-\frac{1}{2}\int\int_{|x-y|\leq\frac{1}{2}}\frac{|x-y|^{-\alpha}-1}{\alpha}F\big(te_0(y)\big)F\big(te_0(x)\big){\rm d}x{\rm d}y\\
&\leq\frac{t^2}{2}\|e_0\|^2
-\frac{1}{2}\int\int_{|x-y|\leq\frac{1}{2}}\ln\frac{1}{|x-y|}F\big(te_0(y)\big)F\big(te_0(x)\big){\rm d}x{\rm d}y\\
&\leq\frac{t^2}{2}\|e_0\|^2
-\frac{\ln2}{2}\bigg(\int_{\R^2}F(te_0){\rm d}x\bigg)^2\\
&\leq\frac{t^2}{2}\|e_0\|^2
-\frac{\ln2}{2}\bigg(\int_{\R^2}F(e_0){\rm d}x\bigg)^2t^{\frac{2}{1-\tau}},\\
\endaligned
\end{equation}
which implies that there exists $t_0>0$ large enough such that $I_{\alpha}(t_0e_0)<0$.
\end{proof}

\noindent So $I_\al$ has a mountain pass geometry, with mountain pass value
$$
c_{\alpha}=\inf\limits_{\gamma\in\Gamma}\max\limits_{t\in[0,1]}I_\al(\gamma(t)),
$$
where
$$
\Gamma:=\{\gamma\in C([0,1], H^1_r(\mathbb{R}^2)):\,\gamma(0)=0,\gamma(1)=e\}.
$$
The existence of a Cerami sequence for $I_\al$, namely $\{u_n\}\subset H^1_r(\R^2)$ such that
$$
I_{\alpha}(u_n)\rightarrow c_{\alpha},\quad\quad (1+\|u_n\|)I'_\alpha(u_n)\rightarrow0,\quad \text{ as } n\rightarrow+\infty,
$$
is given by the variant of the Mountain Pass Theorem in \cite{Ekeland}.
\begin{remark}\label{rem:5.3}
Observe from Lemma \ref{Lem:MP1} that there exist two constants $a,b>0$ independent of $\alpha$ such that $a<c_\al<b$.
\end{remark}

\subsection{Level estimates for the modified energy}
\noindent Let us now define Moser's type functions $w_n(x)$ supported in $B_{\rho}(0)$ as follows:
$$
\aligned
w_n(x)=\frac{1}{\sqrt{2\pi}}\left\{
  \begin{array}{ll}
    \sqrt{\ln n},&  0\leq|x|\leq\rho/n,\\
    \frac{\ln(\rho/|x|)}{ \sqrt{\ln n}}, & \rho/n\leq|x|\leq\rho,\\
    0,&|x|\geq\rho,
  \end{array}
\right.
\endaligned
$$
where $\rho<\frac{1}{4}$ is given in $(f_4)$.
We have
\begin{equation}\label{eqn:moserf1}
\aligned
\|w_n\|^2&= \int_{B_\rho(0)}\bigg(|\nabla w_n|^2+w_n^2\bigg){\rm d}x\\
&=1+\rho^2(\frac{1}{4\ln n}-\frac{1}{4n^2\ln n}-\frac{1}{2n^2})\\
&=:1+\rho^2\delta_n.
\endaligned
\end{equation}
From
$$
\frac{s^{-\alpha}-1}{\alpha}\geq \ln\frac{1}{s},\quad \text{ for }\, s,\al\in(0,1],
$$
the following holds
\begin{equation}\label{eqn:moserf3}
\aligned
&\frac{1}{2}\int_{\R^2}\bigg[G_{\alpha}(x)\ast F(w_n)\bigg] F(w_n){\rm d}x\\
&=\frac{1}{2}\int_{\R^2}\int_{\R^2}\frac{|x-y|^{-\al}-1}{\al} F(w_n(x))F(w_n(y)){\rm d}x{\rm d}y\\
&=\frac{1}{2}\int_{B_\rho(0)}\int_{B_\rho(0)}\frac{|x-y|^{-\al}-1}{\al} F(w_n(x))F(w_n(y)){\rm d}x{\rm d}y\\
&\geq\frac{1}{2}\int_{B_\rho(0)}\int_{B_\rho(0)}\ln\frac{1}{|x-y|} F(w_n(x))F(w_n(y)){\rm d}x{\rm d}y\geq0.
\endaligned
\end{equation}

\begin{lemma}\label{Lem:upperbound}
The mountain pass level $c_\al$ satisfies $\displaystyle\sup_{\al\in(0,1)}c_{\al}<\frac{1}{2}$.
\end{lemma}
\begin{proof}
Recalling ($f_4$), for
\begin{equation}\label{eqn:case0-}
\e\in\bigg(0,\beta-\frac{1}{\rho^2\sqrt{\ln2}\cdot\pi^{3/2}}\bigg),
\end{equation}
there exists $t_\e>0$ such that
\begin{equation}\label{eqn:case0}
\aligned
tF(t)\geq (\beta-\e)e^{4\pi t^2},\quad \text{for}\,\, t\geq t_\e.
\endaligned
\end{equation}
We divide the proof into three cases:

\noindent \emph{Case 1.} Let $t\in\left[0,\sqrt{\frac12}\right]$, then by (\ref{eqn:moserf1})-(\ref{eqn:moserf3}), we have for large $n$,
\begin{equation}\label{eqn:case1}
\aligned
I_{\al}(tw_n)&=\frac{t^2}{2}\|w_n\|^2-\frac{1}{2}\int_{\R^2}\int_{\R^2}
\frac{|x-y|^{-\alpha}-1}{\alpha}F(tw_n(y))F(tw_n(x)){\rm d}y{\rm d}x\\
&<\frac{t^2}{2}\|w_n\|^2< \frac{1}{2}.
\endaligned
\end{equation}

\medskip

\noindent \emph{Case 2.} Let $t\in\bigg(\sqrt{2},+\infty\bigg)$. According to the definition of $w_n$, we have for
large $n$, $tw_n(x)\geq t_\e$ for $x\in B_{\rho/n}$.
From (\ref{eqn:moserf1})-(\ref{eqn:case0}), we deduce that for large $n$,
\begin{equation}\label{eqn:case2}
\aligned
I_{\al}(tw_n) \leq & \frac{t^2}{2}\|w_n\|^2-\frac{1}{2}\int_{\R^2}\int_{\R^2}
\ln\frac{1}{|x-y|}F(tw_n(y))F(tw_n(x)){\rm d}y{\rm d}x\\
\leq&\frac{t^2}{2}\|w_n\|^2-\frac{\ln2}{2} \bigg(\int_{\R^2}F(tw_n){\rm d}x\bigg)^2\\
<&\frac{1+\delta_n\rho^2}{2}t^2-\frac{(\ln 2)\pi^3(\beta-\e)^2\rho^4}{n^4t^{2}\ln n}e^{4 t^2\ln n}\\
<&t^2\bigg(1-\frac{1}{t^{4}\ln n}n^{4 (t^2-1)}\bigg).
\endaligned
\end{equation}
Let
$$
g(n,t):=\frac{1}{t^{4}\ln n}n^{4 (t^2-1)}, t\geq\sqrt{2}, n\geq2.
$$
then there exists $n_0>0$ such that
$g(n,t)\geq 1$ for $n\geq n_0$ and $t\geq \sqrt{2}$.
Hence, when $t\in\bigg(\sqrt{2},+\infty\bigg)$ and $n$ is large enough,
\begin{equation}\label{eqn:case3-2}
\aligned
I_{\al}(tw_n)<\frac{1}{2}.
\endaligned
\end{equation}


\noindent \emph{Case 3.} Let $t\in\left[\sqrt{\frac{1}{2}},\sqrt{2}\right]$. According to the definition of $w_n$, we have for
sufficiently large $n$, $tw_n(x)\geq t_\e$ for $x\in B_{\rho/n}$. From (\ref{eqn:moserf1})--(\ref{eqn:case0}) we have
\begin{equation}\label{eqn:case4}
\aligned
I_{\al}(tw_n) \leq&\frac{t^2}{2}\|w_n\|^2-\frac{1}{2}\int_{\R^2}\int_{\R^2}
\ln\frac{1}{|x-y|}F(tw_n(y))F(tw_n(x)){\rm d}y{\rm d}x\\
<&\frac{1+\delta_n\rho^2}{2}t^2-\frac{(\ln 2)\pi^3(\beta-\e)^2\rho^4}{n^4t^{2}\ln n}e^{4 t^2\ln n}\\
\leq&\frac{1+\delta_n\rho^2}{2}t^2-\frac{1}{2\ln n}n^{4 (t^2-1)}:=\psi_n(t).
\endaligned
\end{equation}
Then, there exists $t_n>0$ such that $\displaystyle\psi_n(t_n)=\max_{t>0}\psi_n(t)$ and
\begin{equation}\label{eqn:case5}
t_n^2=\bigg[1+\frac{\ln(1+\rho^2\delta_n)
-\ln4}{4\ln n}\bigg].
\end{equation}
Clearly, $t_n\in \big[\sqrt{\frac{1}{2}},\sqrt{2}\big]$
for large $n$. Then, by (\ref{eqn:case4}) and (\ref{eqn:case5}) we have
\begin{equation}\label{eqn:case6}
\aligned
\psi_n(t)\leq& \psi_n(t_n)\\
=& t_n^2\frac{1+\delta_n\rho^2}{2}-\frac{1}{8\ln n}(1+\delta_n\rho^2)\\
=&(1+\delta_n\rho^2)\bigg[\frac12+\frac{\ln(1+\rho^2\delta_n)
-\ln4}{8\ln n}-\frac{1}{8\ln n}\bigg]\\
\leq&\bigg(1+\frac{\rho^2}{4\ln n}\bigg)\bigg[\frac12+\frac{\ln(1+\rho^2\delta_n)
}{8\ln n}-\frac{1+\ln4}{8\ln n}\bigg],\\
\endaligned
\end{equation}
 which, together with the definition of $\delta_n$ and the fact $\rho^2<\ln 4$, implies for sufficiently large $n$
\begin{equation}\label{eqn:case7}
\psi_n(t)<\frac{1}{2}.
\end{equation}
Combining (\ref{eqn:case1}), (\ref{eqn:case3-2}) and (\ref{eqn:case7}),
we have $I_{\al}(tw_n)<\frac{1}{2}$. Indeed, for fixed $n$ large enough, there exists
$t_0>0$ such that $I_{\al}(t_0w_n)<0$. Define
$\gamma(t)=tt_0w_n$ for $t\in[0,1]$, then $\gamma\in \Gamma$ and the conclusion follows.
\end{proof}
\subsection{Compactness} In this Section we analyze the behavior of Cerami's sequences. Let us begin with the following
\begin{lemma}\label{Lem:c-bounded}
Assume $(f_1)$--$(f_4)$ and let $\{u_n\}\subset H^1_r(\R^2)$ be an arbitrary Cerami sequence for $I_{\alpha}$ at level $c_\alpha$.
Then, $\{u_n\}$ stays bounded in $H^1_r(\R^2)$ as well as
$$
\bigg|\int_{\R^2}\bigg[G_\alpha(x)\ast F(u_n)\bigg]F(u_n){\rm d}x\bigg|<C,\quad\quad\bigg|\int_{\R^2}\bigg[G_\alpha(x)\ast F(u_n)\bigg]f(u_n)u_n{\rm d}x\bigg|<C.
$$
\end{lemma}
\begin{proof}
Since $\{u_n\}\subset H^{1}_r(\R^2)$ is a Cerami sequence, as $n\rightarrow\infty$ we have
\begin{equation}\label{eqn:bd0}
\frac{1}{2}\|u_n\|^2-\frac{1}{2}\int_{\R^2}\bigg[G_\alpha(x)\ast F(u_n)\bigg]F(u_n){\rm d}x\rightarrow c_\alpha
\end{equation}
and for all $v\in H^1_r(\R^2)$,
\begin{equation}\label{eqn:bd1}
\aligned
\int_{\R^2}\nabla u_n\nabla v{\rm d}x+\int_{\R^2}u_nv{\rm d}x-\int_{\R^2}\bigg[G_\alpha(x)\ast F(u_n)\bigg]f(u_n)v{\rm d}x=o_n(1)\|v\|.
\endaligned
\end{equation}
Now take $v=u_n$ to get
\begin{equation}\label{eqn:bd2}
\|u_n\|^2
-\int_{\R^2}\bigg[G_\alpha(x)\ast F(u_n)\bigg]f(u_n)u_n{\rm d}x=o_n(1).
\end{equation}
In order to verify the boundedness of $\{u_n\}$, let us introduce a suitable test
function as follows
$$
\aligned
v_n:=\left\{
  \begin{array}{ll}
    \frac{F(u_n)}{f(u_n)},&  u_n>0,\\
    \\
    (1-\tau)u_n, & u_n\leq0,\\
  \end{array}
\right.
\endaligned
$$
with $\tau$ as $(f_2)$. We have $|v_n|\leq C|u_n|$ since $F(t)\leq(1-\tau)f(t)t$ and $f(t)=0$ if and only if $t=0$. Thus $v_n$ is well defined in $H^{1}_r(\R^2)$. Taking $v=v_n$ in (\ref{eqn:bd1}) and recalling that $f(t),F(t)=0$ for $t\leq0$ we have
\begin{equation}\label{eqn:bd4}
\aligned
(1-\tau)&\int_{u_n\leq0}|\nabla u_n|^2{\rm d}x+\int_{u_n>0}|\nabla u_n|^2\bigg(1-\frac{F(u_n)f'(u_n)}{f^2(u_n)}\bigg){\rm d}x
+(1-\tau)\int_{u_n\leq0}u_n^2{\rm d}x\\
&+\int_{\R^2}u_n\frac{F(u_n)}{f(u_n)}{\rm d}x-\int_{\R^2}\bigg[G_\alpha(|x|)\ast F(u_n)\bigg]F(u_n){\rm d}x=o_n(1)\|v_n\|.
\endaligned
\end{equation}
and by recalling (\ref{eqn:bd0}) we also have
$$
\aligned
(1-\tau)\int_{\R^2}&|\nabla u_n|{\rm d}x+(1-\tau)\int_{\R^2}u_n^2{\rm d}x+2c_\al-\|u_n\|^2\geq o_n(1)\|v_n\|,
\endaligned
$$
which implies
\begin{equation}\label{eqn:bd6}
\tau\|u_n\|^2\leq o_n(1)\|u_n\|+2c_\alpha.
\end{equation}
As a consequence, we have proved that $\|u_n\|\leq C$ for some $C>0$ independent of $n$. Moreover, we immediately have from (\ref{eqn:bd0}) and (\ref{eqn:bd2})
$$
\bigg|\int_{\R^2}\bigg[G_\alpha(x)\ast F(u_n)\bigg]F(u_n){\rm d}x\bigg|<C,\quad\quad\bigg|\int_{\R^2}\bigg[G_\alpha(x)\ast F(u_n)\bigg]f(u_n)u_n{\rm d}x\bigg|<C.
$$
\end{proof}

\begin{lemma}\label{Lem:compact}
let $\{u_n\}$ be a bounded Cerami sequence for $I_{\alpha}$ at level $c_\al$.
Then, there exists $C>0$ independent of $n$ such that
$$
\int_{\R^2}f(u_n)u_n{\rm d}x\leq C\quad \text{and}\quad \int_{\R^2}F(u_n)^\kappa{\rm d}x\leq C,
$$
where $\kappa\in[1,\frac{1}{2a})$ with $0<a<\frac{1}{2}$ as in Remark \ref{rem:5.3}.
\end{lemma}
\begin{proof}
Since $\{u_n\}$ is bounded in $H^1_r(\R^2)$, we may assume up to a subsequence
$u_n\rightharpoonup u$ in $H_r^1(\R^2)$, $u_n\rightarrow u$
in $L_{loc}^s(\R^2)$ for any $2\leq s<+\infty$ and $u_n\rightarrow u$ a.e. in $\R^2$, for which
$$
\lim\limits_{n\rightarrow+\infty}\|u_n\|^2=A^2\geq\|u\|^2.
$$
Let us introduce the following auxiliary function
$$
H(t)=\int_{0}^t\frac{\sqrt{F(s)f'(s)}}{f(s)}{\rm d}s,
$$
and define $v_n:=H(u_n)$. Let us show that
\begin{equation}\label{eqn:keji1}
\|v_n\|^2\leq 1.
\end{equation}
From
$$
\int_{\R^2}\bigg[G_\alpha(|x|)\ast F(u_n)\bigg]F(u_n){\rm d}x=A^2-2c_\al
$$
and
$$
\int_{\R^2}|\nabla u_{n}|^2\bigg(1-\frac{F(u_n)f'(u_n)}{f^2(u_n)}\bigg)-\int_{\R^2}\bigg[G_\alpha(|x|)\ast F(u_n)\bigg]F(u_n){\rm d}x+\int_{\R^2}u_n\frac{F(u_n)}{f(u_n)}{\rm d}x=o_n(1),
$$
we have
$$
\aligned
\int_{\R^2}|\nabla u_n|^2\bigg(1-\frac{F(u_n)f'(u_n)}{f^2(u_n)}\bigg){\rm d}x
+\int_{\R^2}u_n\frac{F(u_n)}{f(u_n)}{\rm d}x+2c_\al-\|u_n\|^2=o_n(1),
\endaligned
$$
and in turn
$$
\aligned
\|v_n\|^2&=\int_{\R^2}|\nabla H(u_n)|^2{\rm d}x+\int_{\R^2}H^2(u_n){\rm d}x\\
&=2c_\al+\int_{\R^2}\bigg(u_n\frac{F(u_n)}{f(u_n)}-u_n^2+H^2(u_n)\bigg)+o_n(1)\\
&\leq 2c_\al<1
\endaligned
$$
for $n$ large enough. Next we give an $L^1$-estimate of the sequence $\{f(u_n)u_n\}$
by using the estimate for $\{\|v_n\|\}$. By $(f_3)$, for any $\e>0$, there exists $t_\e>0$ such that
 $$
 \frac{\sqrt{F(t)f'(t)}}{f(t)}\in[1-\e,1+\e],\quad\text{for\,all\,} t\geq t_\e.
 $$
 By $(f_2)$ we have
 \begin{equation}\label{eqn:keji2}
 v_n\geq \int_{0}^{t_\e}\tau {\rm d}t+\int^{u_n}_{t_\e}(1-\e){\rm d}t\geq(1-\e)(u_n-t_\e),
\end{equation}
 which implies
 $$
 u_n\leq t_\e+\frac{v_n}{1-\e}
 $$
 for any $x\in\R^2$. Hence, by ($f_1$) we have that for any given $\e>0$ above,
 there exists $C_\e$ such that
 \begin{equation}\label{eqn:keji3}
\aligned
\int_{\R^2}f(u_n)u_n{\rm d}x&\leq \int_{|u_n|\leq t_\e} f(u_n)u_n{\rm d}x+\int_{|u_n|\geq t_\e}f(u_n)u_n{\rm d}x\\
&\leq C_\e\|u_n\|^{2}+\int_{|u_n|\geq t_\e}f\bigg(t_\e+\frac{v_n}{1-\e}\bigg)\bigg(t_\e+\frac{v_n}{1-\e}\bigg){\rm d}x\\
&\leq C_\e\|u_n\|^{2}+C_\e\int_{|u_n|\geq t_\e} e^{4\pi\big(t_\e+\frac{v_n}{1-\e}\big)^{2}}\bigg(t_\e+\frac{v_n}{1-\e}\bigg)^{p+1}{\rm d}x\\
&\leq C_\e\|u_n\|^{2}+C_\e\int_{|u_n|\geq t_\e} e^{4\pi(1+\e)\big(t_\e+\frac{v_n}{1-\e}\big)^{2}}{\rm d}x,
\endaligned
\end{equation}
where in the last inequality we use the fact that for large values of $u_n$,
also $v_n$ is large such that
$$
\bigg(t_\e+\frac{v_n}{1-\e}\bigg)^{p+1}\leq C_\e e^{4\pi\e\big(t_\e+\frac{v_n}{1-\e}\big)^{2}}.
$$
In view of (\ref{eqn:keji2}), $v_n\geq \tau t_\e$ if $u_n\geq t_\e$, and then it follows from
(\ref{eqn:keji3}) that
 \begin{equation}\label{eqn:keji4}
\aligned
\int_{\R^2}f(u_n)u_n{\rm d}x&\leq C_\e\|u_n\|^{2}+C_\e\int_{|u_n|\geq t_\e} e^{4\pi(1+\e)^2\frac{v_n^2}{(1-\e)^2}}{\rm d}x\\
&\leq C_\e\|u_n\|^{2}+C_\e\int_{\R^2} e^{4\pi(1+\e)^2\frac{v_n^2}{(1-\e)^2}}-1{\rm d}x\\
&\leq C_\e\|u_n\|^{2}+C_\e\int_{\R^2} e^{4\pi(1+\e)^2\|v_n\|^2\frac{v_n^2}{\|v_n\|^2(1-\e)^2}}-1{\rm d}x.
\endaligned
\end{equation}
Since $\|v_n\|^2\leq 2c_\al+o_n(1)$, $\|v_n\|^2\leq 2c_\al+\sigma<1$ for $n$ large enough and
$\sigma>0$ is sufficiently small which is also independent of $\al$. Finally,
the following holds
$$
\frac{(1+\e)^2}{(1-\e)^2}\|v_n\|^2\leq \frac{(1+\e)^2}{(1-\e)^2}(2c_\alpha+\sigma)<1
$$
for $\e>0$ small enough. As a consequence, from (\ref{eqn:keji4}) and ($f_2$) we get
$$
\int_{\R^2}f(u_n)u_n{\rm d}x\leq C\quad \text{and\,then}\quad \int_{\R^2}F(u_n){\rm d}x\leq C
$$
for some $C$ independent of $n$. The remain case $\kappa\in(1,\frac{1}{2a})$
has been proven in \cite{Cassani21}, that is,
$$
\int_{\R^2}F(u_n)^\kappa{\rm d}x\leq C.
$$
Combining the above facts yields the proof.
\end{proof}
\noindent Since $\{u_n\}$ is bounded in $H_r^1(\R^2)$, up to a subsequence still denoted by $\{u_n\}$, there exists $u_0\in H_r^1(\R^2)$ such that
\begin{equation}\label{eqn:fun2-}
u_n\rightharpoonup u_0\,\,\text{ weakly in } H_r^1(\R^2),\quad\quad u_n\rightarrow u_0\,\,\text{ strongly in } L^p(\R^2),\,\,p\in(2,+\infty).
\end{equation}
\begin{lemma}\label{Lem:compact0}
Assume $(f_1)$--$(f_4)$ and let $\{u_n\}$ be a bounded Cerami sequence for $I_{\alpha}$ at level $c_\al$.
Then,
\begin{equation}\label{eqn:cl2}
\int_{\R^2}\bigg[\frac{1}{|x|^\alpha}\ast F(u_n)\bigg]F(u_n){\rm d}x\rightarrow\int_{\R^2}\bigg[\frac{1}{|x|^\alpha}\ast F(u_0)\bigg]F(u_0){\rm d}x
\end{equation}
and
\begin{equation}\label{eqn:cl2++}
\int_{\R^2}F(u_n){\rm d}x\rightarrow\int_{\R^2}F(u_0){\rm d}x,\quad\text{as}\,\,n\rightarrow\infty.
\end{equation}
\end{lemma}
\begin{proof}
From $I'_\al(u_n)u_n=o_n(1)$ we obtain the following
$$
\aligned
&\|u_n\|^2+\frac{1}{\al}\int_{\R^2}F(u_n){\rm d}x\int_{\R^2}f(u_n)u_n{\rm d}x\\
&=2c_\al+
\frac{1}{\al}\int_{\R^2}\bigg[\frac{1}{|x|^\al}\ast F(u_n)\bigg]f(u_n)u_n{\rm d}x,
\endaligned
$$
which implies immediately by Lemma \ref{Lem:c-bounded} and Lemma \ref{Lem:upperbound} that
there exists $C>0$ independent of $n$ such that
\begin{equation}\label{eqn:fun2++}
\aligned
\frac{1}{\al}\int_{\R^2}\bigg[\frac{1}{|x|^\al}\ast F(u_n)\bigg]f(u_n)u_n{\rm d}x\leq C.
\endaligned
\end{equation}
From (\ref{eqn:f5}) and (\ref{eqn:fun2++}) we deduce that for any $\e>0$, there exists $M_\e>0$ such that
\begin{equation}\label{eqn:cl3+0}
\int_{|u_n|\geq M_\e}\bigg[\frac{1}{|x|^\al}\ast F(u_n)\bigg]F(u_n){\rm d}x\leq\frac{M_0}{M_\e}\int_{|u_n|\geq M_\e}\bigg[\frac{1}{|x|^\al}\ast F(u_n)\bigg]f(u_n)u_n{\rm d}x\leq \e.
\end{equation}
and
\begin{equation}\label{eqn:cl3}
\aligned
\int_{|u_0|\geq M_\e}\bigg[\frac{1}{|x|^\al}\ast F(u_0)\bigg]F(u_0){\rm d}x
\leq \e.
\endaligned
\end{equation}
Observe that from Lemma \ref{Lem:c-bounded} and H\"{o}lder's inequality we have
\begin{equation}\label{eqn:fun3+0}
\aligned
&\int_{\R^2}\frac{1}{|x-y|^\al}F(u_n(y)){\rm d}y\\
&\leq\int_{|x-y|\leq1}\frac{1}{|x-y|^\al}F(u_n(y)){\rm d}y+\int_{|x-y|\geq1}\frac{1}{|x-y|^\al}F(u_n(y)){\rm d}y\\
&\leq C+\bigg(\int_{|x-y|\leq1}\frac{1}{|x-y|}dy\bigg)^{\al}\cdot\bigg( \int_{\R^2} F(u_n(y))^{\frac{1}{1-\al}}{\rm d}y\bigg)^{1-\al}\\
&\leq C+C\bigg( \int_{\R^2} F(u_n(y))^{\frac{1}{1-\al}}{\rm d}y\bigg)^{1-\al}\leq C,
\endaligned
\end{equation}
where in the above inequality we have taken $\al\in\bigg(0,1-\frac{1}{\kappa}\bigg)$.
Let us define the following sequence of functions
\begin{equation}\label{eqn:cl3+}
\aligned
G(x,u_n)&:=\bigg[\frac{1}{|x|^\al}\ast F(u_n)\bigg]\bigg(\e|u_n|^2+C_\e|u_n|^q\bigg)\\
&\geq\bigg[\frac{1}{|x|^\al}\ast F(u_n)\bigg]\bigg(f(u_n)u_n\chi_{|u_n|<M_\e}\bigg),
\endaligned
\end{equation}
where $C_\e>0$ depends only on $\e$ and $q>2$. Moreover, using the Hardy-Littlewood-Sobolev inequality and (\ref{eqn:fun3+0}) we deduce that
\begin{equation}\label{eqn:cl4}
\aligned
&\bigg|\int_{\R^2}G(x,u_n)-G(x,u_0){\rm d}x\bigg|\\
\leq& C\e+C_\e\bigg|\int_{\R^2}\bigg[\frac{1}{|x|^\al}\ast F(u_n)\bigg]\bigg(|u_n|^q-|u_0|^q\bigg){\rm d}x\bigg|\\
&+C_\e\bigg|\int_{\R^2}\bigg[\frac{1}{|x|^\al}\ast \big(F(u_n)- F(u_0)\big)\bigg]|u_0|^q{\rm d}x\bigg|\\
\leq&C\e+ C_\e C\|F(u_n)\|_{\frac{4}{4-\al}}\||u_n|^q-|u_0|^q\|_{\frac{4}{4-\al}}\\
&+C_\e\int_{\R^2}\bigg[\frac{1}{|x|^\al}\ast \big(F(u_n)- F(u_0)\big)\bigg]|u_0|^q{\rm d}x\\
\leq&C\e+ C_\e C \cdot o_n(1)+C_\e\int_{\R^2}\bigg[\frac{1}{|x|^\al}\ast \big(F(u_n)- F(u_0)\big)\bigg]|u_0|^q{\rm d}x\\
:=& C\e+C_\e C \cdot o_n(1)+D_1,
\endaligned
\end{equation}
where in the above inequality we have the fact that $\|F(u_n)\|_{\frac{4}{4-\al}}$ is uniformly bounded by taking $\al$ small enough.
There exists $R>0$ such that
\begin{equation}\label{eqn:cl5}
\aligned
D_1&:=E_1+E_2:=
C_\e\int_{|x|\geq R_\e}\bigg[\frac{1}{|x|^\al}\ast \big(F(u_n)- F(u_0)\big)\bigg]|u_0|^q{\rm d}x\\
&+C_\e\int_{|x|\leq R_\e}\bigg[\frac{1}{|x|^\al}\ast \big(F(u_n)- F(u_0)\big)\bigg]|u_0|^q{\rm d}x,
\endaligned
\end{equation}
where for any fixed $\e>0$, by taking $R=R_\e>0$ large enough, it follows from (\ref{eqn:fun3+0}) that
$$
|E_1|\leq C_\e C\int_{|x|\geq R_\e} |u_0|^q{\rm d}x<\e.
$$
Moreover, by virtue of (\ref{eqn:fun3+0}), we employ the Lebesgue dominated convergence theorem to deduce that
$$
E_2=C_\e\int_{|x|\leq R_\e}\bigg[\frac{1}{|x|^\al}\ast \big(F(u_n)- F(u_0)\big)\bigg]|u_0|^q{\rm d}x=C_\e\cdot o_n(1).
$$
Based on the above facts, combining (\ref{eqn:cl5}) and (\ref{eqn:cl4}), one has
$$
\bigg|\int_{\R^2}G(x,u_n)-G(x,u_0){\rm d}x\bigg|=o_n(1),
$$
that is,
the control function sequence $\{G(x,u_n)\}$ has a strong convergence subsequence in $L^1(\R^2)$.
Hence, using the Lebesgue dominated convergence theorem once again, from (\ref{eqn:cl3+}) we immediately obtain
$$
\aligned
&\int_{\R^2}\bigg[\frac{1}{|x|^\al}\ast F(u_n)\bigg]f(u_n)u_n\chi_{|u_n|<M_\e}{\rm d}x\\
&\rightarrow\int_{\R^2}\bigg[\frac{1}{|x|^\al}\ast F(u_0)\bigg]f(u_0)u_0\chi_{|u_0|<M_\e}{\rm d}x,\quad \text{ as } n\rightarrow\infty,
\endaligned
$$
which, together with (\ref{eqn:cl3+0}) and (\ref{eqn:cl3}), implies (\ref{eqn:cl2}). Analogously we also have
$$
\int_{\R^2}F(u_n){\rm d}x\rightarrow\int_{\R^2}F(u_0){\rm d}x,\quad{as}\,\,n\rightarrow\infty.
$$
\end{proof}

\begin{lemma}\label{Lem:compact++}
Assume that $(f_1)$--$(f_4)$ and let $\{u_n\}$ be a bounded Cerami sequence for $I_{\alpha}$ at level $c_\al$.
Then, there exists a nontrivial $u_0\in H_r^1(\R^2)$ such that $u_n\rightarrow u_0$ in $H_r^1(\R^2)$, as $n\to\infty$.
\end{lemma}
\begin{proof}
We first \textit{claim} that for any $\varphi\in C_0^\infty(\R^2)$
\begin{equation}\label{eqn:fun3}
\aligned
\frac{1}{\al}\int_{\R^2}\bigg[\frac{1}{|x|^\al}\ast F(u_n)\bigg]f(u_n)\varphi{\rm d}x
\rightarrow \frac{1}{\al}\int_{\R^2}\bigg[\frac{1}{|x|^\al}\ast F(u_0)\bigg]f(u_0)\varphi{\rm d}x, \quad \text{as}\,\, n\rightarrow\infty.
\endaligned
\end{equation}
Indeed, let us define the sequence of functions
$$
g(x,u_n):=\bigg[\frac{1}{|x|^\al}\ast F(u_n)\bigg]f(u_n)
$$
restricted to any compact domain $\Omega$.  Hence, (\ref{eqn:fun3+0}) implies
 $g(x,u_n)\leq Cf(u_n)$ for $x\in \Omega$. Moreover, using $(f_1)$--$(f_2)$ and Lemma \ref{Lem:c-bounded} we have that there exists $M>0$ such that
$$
\aligned
&\int_{\Omega}g(x,u_n){\rm d}x\leq \int_{\Omega}Cf(u_n){\rm d}x\\
&\leq \int_{\Omega\cap\{x|\,|u_n|\leq M\}}Cf(u_n){\rm d}x+\int_{\Omega\cap\{x|\,|u_n|\geq M\}}Cf(u_n){\rm d}x\\
&\leq C+\frac{C}{M}\int_{\Omega\cap\{x|\,|u_n|\geq M\}}f(u_n)u_n{\rm d}x\leq C.
\endaligned
$$
Then $g\in L^1(\Omega)$. Thanks to (\ref{eqn:fun2++}), using similar arguments as that in Lemma 2.1 of \cite{Figueiredo95}, the claim follows. Similarly, we can also prove
\begin{equation}\label{eqn:fun3+f}
\aligned
\frac{1}{\al}\int_{\R^2}f(u_n){\rm d}x\rightarrow \frac{1}{\al}\int_{\R^2}f(u_0){\rm d}x,\quad\text{in}\,\,L_{loc}^1(\R^2),\quad \text{as}\,\,n\rightarrow\infty.
\endaligned
\end{equation}
It follows from (\ref{eqn:fun3}), (\ref{eqn:fun3+f}), and Lemma \ref{Lem:compact0} that for any $\varphi\in C_0^\infty(\R^2)$,
\begin{equation}\label{eqn:con8}
I'_\al(u_n)\varphi\rightarrow
I'_\al(u_0)\varphi,\quad \text{as}\,\,n\rightarrow\infty.
\end{equation}
That is, $I'_\al(u_0)=0$. Let us use the following
$$
\aligned
v_0=\left\{
  \begin{array}{ll}
    \frac{F(u_0)}{f(u_0)},&  u_n>0,\\
    \\
    (1-\tau)u_0, & u_0\leq0,\\
  \end{array}
\right.
\endaligned
$$
in place of $v_n$ in Lemma \ref{Lem:c-bounded}, to obtain in a similar fashion
\begin{equation}\label{eqn:fun4}
\aligned
2I_\al(u_0)\geq \tau\|u_0\|^2\geq0.
\endaligned
\end{equation}


\noindent Next we distinguish two cases:

\medskip

\noindent \emph{Case 1.} $u_0\equiv0$. Recalling Lemma \ref{Lem:compact0}, one has
$$
\frac{1}{2}> c_\al= I_\al(u_n)+o_n(1)=\frac{1}{2}\|u_n\|^2+o_n(1).
 $$
 Then there exists $\epsilon_0>0$ sufficiently small such that for large $n$
\begin{equation}\label{eqn:c-bound4}
\|u_n\|^2<(1-\epsilon_0),
\end{equation}
and then there exists $\theta\in(1,2)$ such that
$$
(1+\epsilon_0)(1-\epsilon_0)\theta<1.
$$
In conclusion, it follows from $(f_1)$--$(f_3)$ and (\ref{eqn:c-bound4}) that for any $\xi>0$, there exists $C_\xi>0$ such that
\begin{equation}\label{eqn:con9}
\aligned
&\int_{\R^2}f(u_n)u_n{\rm d}x\\
&\leq \int_{\R^2}\bigg(\xi u_n^2+C_\xi |u_n|^{p+1}e^{4\pi\|u_n\|^2\frac{|u_n|^2}{\|u_n\|^2}}\bigg){\rm d}x\\
&\leq \xi \|u_n\|^2+C_\xi \bigg(\int_{\R^2}|u_n|^{\theta'(p+1)}{\rm d}x\bigg)^{\frac{1}{\theta'}}
\bigg(\int_{\R^2}e^{4\pi(1-\e_0)\theta\frac{|u_n|^2}{\|u_n\|^2}}{\rm d}x\bigg)^{\frac{1}{\theta}}\\
&\leq \xi \|u_n\|^2+C_\xi C\bigg(\int_{\R^2}|u_n|^{\theta'(p+1)}{\rm d}x\bigg)^{\frac{1}{\theta'}}
\bigg(\int_{\R^2}e^{4\pi(1+\e_0)(1-\e_0)\theta\frac{|u_n|^2}{\|u_n\|^2}}-1{\rm d}x\bigg)^{\frac{1}{\theta}},
\endaligned
\end{equation}
where $\frac{1}{\theta}+\frac{1}{\theta'}=1$.
Hence, combining (\ref{eqn:fun2-}) and (\ref{eqn:con9}), together with the
arbitrariness of $\xi$, yields
\begin{equation}\label{eqn:con10}
\int_{\R^2}f(u_n)u_n{\rm d}x=o_n(1),
\end{equation}
which implies by (\ref{eqn:fun3+0}) the following
$$
\int_{\R^2}\bigg[\frac{1}{|x|^\alpha}\ast F(u_n)\bigg]f(u_n)u_n{\rm d}x=o_n(1).
$$
From $I'_\al(u_n)u_n=o_n(1)$ we get $u_n\rightarrow0$ in $H_r^1(\R^2)$, as $n\rightarrow\infty$. This is a contradiction with the fact $I_{\alpha}(u_n)\rightarrow c_{\al}$, as $n\rightarrow\infty$.

\medskip

\noindent \emph{Case 2.} $u_0\not\equiv0$. That is, $\|u_0\|>0$. Next we show that
\begin{equation}\label{eqn:claim1}
\|u_n\|^2\rightarrow \|u_0\|^2,\quad \text{ as }\,\,n\rightarrow\infty.
\end{equation}
By Fatou's lemma and Lemma \ref{Lem:compact0} we have
\begin{equation}\label{eqn:claim2+}
\aligned
 I_{\alpha}(u_0)=&\frac{1}{2}\|u_0\|^2+\frac{1}{2\alpha}\bigg(\int_{\R^2}F(u_0){\rm d}x\bigg)^2-\frac{1}{2\al}\int_{\R^2}\bigg[\frac{1}{|x|^\alpha}\ast F(u_0)\bigg]F(u_0){\rm d}x\\
\leq&\liminf_{n\rightarrow\infty}\bigg(\frac{1}{2}\|u_n\|^2+\frac{1}{2\alpha}\bigg(\int_{\R^2}F(u_n){\rm d}x\bigg)^2
-\frac{1}{2\al}\int_{\R^2}\bigg[\frac{1}{|x|^\alpha}\ast F(u_n)\bigg]F(u_n){\rm d}x\bigg)\\
=& \liminf_{n\rightarrow\infty}I_{\al}(u_n)=c_{\al}.
\endaligned
\end{equation}
If $I_{\alpha}(u_0)=c_{\al}$, by (\ref{eqn:claim2+}) we obtain immediately (\ref{eqn:claim1}).
Otherwise, if $I_{\alpha}(u_0)<c_{\al}$, then
\begin{equation}\label{eqn:claim3}
\aligned
\|u_0\|^2+\frac{1}{\alpha}\bigg(\int_{\R^2}F(u_0){\rm d}x\bigg)^2
<2c_{\al}+\frac{1}{\al}\int_{\R^2}\bigg[\frac{1}{|x|^\alpha}\ast F(u_0)\bigg]F(u_0){\rm d}x.
\endaligned
\end{equation}
 In view of the definition of $I_{\alpha}$, we also have
\begin{equation}\label{eqn:claim4}
\aligned
\lim\limits_{n\rightarrow\infty}\bigg(\|u_n\|^2+\frac{1}{\alpha}\bigg(\int_{\R^2}F(u_n){\rm d}x\bigg)^2\bigg)
=2c_{\al}+\frac{1}{\al}\int_{\R^2}\bigg[\frac{1}{|x|^\alpha}\ast F(u_0)\bigg]F(u_0){\rm d}x.
\endaligned
\end{equation}
Take
$$
w_n=\frac{u_n}{\sqrt{\|u_n\|^2+\frac{1}{\alpha}\bigg(\int_{\R^2}F(u_n){\rm d}x\bigg)^2}}
$$
and
$$
w_0=\frac{u_0}{\sqrt{2c_{\al}+\frac{1}{\al}\int_{\R^2}\bigg[\frac{1}{|x|^\alpha}\ast F(u_0)\bigg]F(u_0){\rm d}x}}.
$$
From (\ref{eqn:claim3}) and (\ref{eqn:claim4}) we have $\|w_n\|\leq1$, $w_n\rightharpoonup w_0$ and
$\|w_0\|<1$. It is not difficult to deduce by \eqref{eqn:claim3}
that $\displaystyle\lim_{n\rightarrow\infty}\|u_n\|^2>\|u_0\|^2$ and $A>\|w_0\|^2$. Following Lions \cite{Lions85}, one has
\begin{equation}\label{eqn:claim4+1}
\sup\limits_{n\in\N}\bigg(\int_{\R^2}e^{4\pi r w_n^2}-1\bigg){\rm d}x<\infty
\end{equation}
for all
$$
 r<\bar{r}:=\frac{1}{B-\|w_0\|^2}=2\frac{c_{\al}+\frac{1}{2\al}\int_{\R^2}\bigg[\frac{1}{|x|^\alpha}\ast F(u_0)\bigg]F(u_0){\rm d}x}{\|u_n\|^2-\|u_0\|^2}+o_n(1),
$$
where $\displaystyle B=\lim_{n\rightarrow\infty}\|w_n\|^2$.
By Lemma \ref{Lem:compact0} and (\ref{eqn:fun4}), the Brezis-Lieb lemma yields
$$
\aligned
I_\al(u_n)-I_\al(u_0)=\frac{1}{2}\bigg(\|u_n\|^2-\|u_0\|^2\bigg)+o_n(1)< c_\al<\frac{1}{2}.
\endaligned
$$
Then, recalling (\ref{eqn:claim4}),
we can always choose $s>1$ sufficiently close to $1$ and $\epsilon>0$ small enough such that
$$
\aligned
&s(1+\epsilon) \bigg(\|u_n\|^2+\frac{1}{\alpha}\bigg(\int_{\R^2}F(u_n){\rm d}x\bigg)^2\bigg)\\
\leq& r<\frac{1}{B-\|w_0\|^2}\\
=&2\frac{c_{\al}
+\frac{1}{2\al}\int_{\R^2}\bigg[\frac{1}{|x|^\alpha}\ast F(u_0)\bigg]F(u_0){\rm d}x}{\|u_n\|^2-\|u_0\|^2}+o_n(1)
\endaligned
$$
for some $r$ satisfying (\ref{eqn:claim4+1}).
Based on the above facts,
it follows from (\ref{eqn:fun3+f}), (\ref{eqn:claim4+1}) and ($f_1$) that
\begin{equation}\label{eqn:claim5}
\aligned
&\bigg|\int_{\R^2}f(u_n)u_n-f(u_0)u_0{\rm d}x\bigg|\\
=&\bigg|\int_{\R^2}f(u_n)(u_n-u_0)-(f(u_n)-f(u_0))u_0{\rm d}x\bigg|\\
\leq&C\bigg(\int_{\R^2}|u_n|^{ps}e^{4s\pi \|u_n\|^2\frac{|u_n|^2}{\|u_n\|^2}}dx\bigg)^{\frac{1}{s}}\bigg(\int_{\R^2}|u_n-u_0|^{\frac{s}{s-1}}dx\bigg)^{\frac{s-1}{s}}
\\&+C\int_{\R^2}|u_n(u_n-u_0)|{\rm d}x
+\int_{\R^2}\big|(f(u_n)-f(u_0))u_0\big|{\rm d}x\rightarrow0,\quad \text{ as }\,\,n\rightarrow\infty.
\endaligned
\end{equation}
On the other hand, from (\ref{eqn:fun3+0}) and (\ref{eqn:fun3}) we deduce that
\begin{equation}\label{eqn:claim6}
\aligned
&\bigg|\int_{\R^2}\bigg[\frac{1}{|x|^\al}\ast F(u_n)\bigg]f(u_n)u_n{\rm d}x-\int_{\R^2}\bigg[\frac{1}{|x|^\al}\ast F(u_0)\bigg]f(u_0)u_0{\rm d}x\bigg|\\
=&\bigg|\int_{\R^2}\bigg[\frac{1}{|x|^\al}\ast F(u_n)\bigg]\big(f(u_n)u_n-f(u_0)u_0\big){\rm d}x\bigg|\\
&+\bigg|\int_{\R^2}\bigg[\frac{1}{|x|^\al}\ast \big(F(u_n)- F(u_0)\big)\bigg]f(u_0)u_0{\rm d}x\bigg|\\
\leq&C\bigg(\int_{\R^2}|u_n|^{ps}e^{4s\pi \|u_n\|^2\frac{|u_n|^2}{\|u_n\|^2}}dx\bigg)^{\frac{1}{s}}\bigg(\int_{\R^2}|u_n-u_0|^{\frac{s}{s-1}}dx\bigg)^{\frac{s-1}{s}}
\\&+C\int_{\R^2}|u_n(u_n-u_0)|{\rm d}x
+\int_{\R^2}\big|(f(u_n)-f(u_0))u_0\big|{\rm d}x\rightarrow 0,\quad \text{ as }\,\,n\rightarrow\infty.
\endaligned
\end{equation}
Combining (\ref{eqn:claim5}) and (\ref{eqn:claim6}) we
obtain $u_n\rightarrow u_0$ in $H_r^1(\R^2)$. Then $I(u_0)<c_\al$ is not true.
Hence, $I'_\al(u_0)=0$ and $I_\al(u_0)=c_\al$.
\end{proof}

\section{Proof of Theorem 1.1}
\noindent By virtue of Lemma \ref{Lem:compact}, we have that $u_\al\in H_r^1(\R^2)$ is a
critical point of $I_\al$ with $I_\al(u_\al)=c_\al$. Recalling Remark \ref{rem:5.3},
we have $a<I_\al(u_\al)<b$ with $a,b>0$ independent of $\al$. Similar arguments as in
Lemma \ref{Lem:c-bounded} yield $u_\al$ bounded in $H_r^1(\R^2)$ uniformly in $\al>0$.
In order to study the limit properties of $u_\al$ as $\al\rightarrow0^+$, we are going
to establish some estimates for $u_\al$.

\noindent We may assume as  $\alpha\rightarrow0^+$, up to a subsequence, the following:
\begin{equation}\label{eqn:limit}
\aligned
& u_\alpha\rightharpoonup u_0 \quad\text{in}\,\, H_r^1(\R^2),\\
&u_\alpha\rightarrow u_0 \quad\,\,a.e.\,\,\text{in}\,\, \R^2,\\
&u_\alpha\rightarrow u_0\quad\text{in}\,\, L^s(\R^2)\quad \text{for}\,\,s\in(2,+\infty).
\endaligned
\end{equation}
\begin{lemma}\label{Lem:jianjin+}
For $\omega>1$, sufficiently close to $1$, there exists $C>0$ independent of $\al\in\bigg(0,\frac{4(\omega-1)}{3\omega}\bigg)$ such that
$$
\left|\int_{|x-y|\leq 1}\frac{F(u_\alpha(y))}{|x-y|^{\frac{4(\omega-1)}{3\omega}}}{\rm d}y\right|\leq C,
$$
and as $|x|\rightarrow\infty$ and $\al\rightarrow0^+$,
$$
 \int_{|x-y|\leq 1}\frac{F(u_\alpha(y))}{|x-y|^{\frac{4(\omega-1)}{3\omega}}}{\rm d}y\rightarrow 0.
$$
\end{lemma}
\begin{proof}
Arguing as in Lemma \ref{Lem:compact}, by taking $v_\al:=H(u_\al)$ and by \eqref{eqn:keji1},
 we deduce
 $$
 \sup_{\al\in(0,\frac{4(\omega-1)}{3\omega})}\|v_\al\|<1
 $$
 and there exists $C>0$
 independent of $\al>0$ such that
 \begin{equation}\label{eqn:al1}
\int_{\R^2}F(u_\al)^\omega{\rm d}x<C.
\end{equation}
Moreover, we know that for any $\e>0$, there exists $t_\e>0$ such that
$$
 u_\al(x)\leq t_\e+\frac{v_\al(x)}{1-\e}\quad\text{for\,any}\,\,x\in\R^2,
 $$
 which implies by Young's inequality that there exists $C_\e>0$ such that
 $$
 u_\al^2\leq C_\e t_\e^2+(1+\e)\frac{v_\al^2}{(1-\e)^2}.
 $$
From $(f_1)$--$(f_3)$ and using the H\"{o}lder inequality, one has
$$
\aligned
&\int_{|x-y|\leq 1}\frac{F(u_\alpha(y))}{|x-y|^{\frac{4(\omega-1)}{3\omega}}}{\rm d}y\\
&\leq\bigg(\int_{|x-y|\leq 1}\frac{1}{|x-y|^{\frac{4}{3}}}{\rm d}y\bigg)^{\frac{\omega}{\omega-1}}\cdot
\bigg(\int_{|x-y|\leq 1}|F(u_\alpha)|^\omega {\rm d}y\bigg)^{\frac{1}{\omega}}\\
&\leq C \bigg(\int_{|x-y|\leq 1}|u_\alpha|^{2\omega}+|u_\alpha|^{p\omega}e^{4\pi\omega(t_\e+\frac{v_n}{1-\e})^2} {\rm d}y\bigg)^{\frac{1}{\omega}}\\
&\leq C C_\e\bigg(\int_{|x-y|\leq 1}|u_\alpha|^{2\omega}+|u_\alpha|^{p\omega}e^{4\pi\omega(1+\e)\frac{v_n^2}{(1-\e)^2}} {\rm d}y\bigg)^{\frac{1}{\omega}}\\
&\leq C \bigg(\int_{|x-y|\leq 1}|u_\alpha|^{2\omega}{\rm d}y+
\big(\int_{|x-y|\leq 1}|u_\alpha|^{p\omega\zeta}{\rm d}y\big)^{1/\zeta}\cdot\big(\int_{\R^2}e^{4\pi\zeta'
\frac{(1+\e)}{(1-\e)^2}\omega\|v_n\|^2\frac{v_n^2}{\|v_n\|^2}}-1{\rm d}x )^{1/\zeta'}\bigg)^{\frac{1}{\omega}}\\
&\leq C \bigg(\int_{|x-y|\leq 1}|u_\alpha|^{2\omega}{\rm d}y+
C\big(\int_{|x-y|\leq 1}|u_\alpha|^{p\omega\zeta}{\rm d}y\big)^{1/\zeta}\bigg)^{\frac{1}{\omega}},
\endaligned
$$
where $\zeta>1$ and $\zeta'=\frac{\zeta}{\zeta-1}$ and in the last inequality
we let $\zeta',\omega$ sufficiently close to $1$.
Thanks to (\ref{eqn:limit}), we obtain the desired result.
\end{proof}

\noindent Next we establish the exponential decay of $u_\al$ at infinity uniformly with respect to $\al$.

\begin{lemma}\label{Lem:R2}
There exist $R,M>0$ (independent of $\al$) such that
$$
u_\alpha(x)\leq M\exp\bigg(-\frac{1}{2}|x|\bigg) \quad\text{for} \,\,|x|\geq R.
$$
\end{lemma}
\begin{proof}
Since $u_\alpha$ is a positive function of equation (\ref{eqn:2-spos}) and by Lemma \ref{Lem:G} we obtain
\begin{equation}\label{eqn:re1}
\aligned
    -\triangle {u_\alpha}+ u_\alpha
     &\leq \int_{|x-y|\leq 1}\frac{|x-y|^{-\alpha}-1}{\alpha}F(u_\alpha(y)){\rm d}y f(u_\alpha){\rm d}x\\
     &\leq C\int_{|x-y|\leq 1}\frac{F(u_\alpha(y))}{|x-y|^{\frac{4(\omega-1)}{3\omega}}}{\rm d}yf(u_\alpha){\rm d}x,
\endaligned
\end{equation}
where $\al\in\bigg(0,\frac{4(\omega-1)}{3\omega}\bigg)$.
In view of Lemma \ref{Lem:jianjin+},
there exist $R_1>0$ and $\al^*\in\bigg(0,\frac{4(\omega-1)}{3\omega}\bigg)$
such that
\begin{equation}\label{eqn:re2}
\int_{|x-y|\leq 1}\frac{F(u_\al(y))}{|x-y|^{\frac{4(\omega-1)}{3\omega}}}{\rm d}y\leq \frac{1}{C}
\end{equation}
 for $|x|\geq R_1$ and $\al\in(0,\al^*)$. By recalling the radial Lemma A.IV in \cite{Berestycki83}, there exists $C>0$ independent of $\alpha$ such that
$$
|u_\alpha(x)|\leq C|x|^{-\frac12}\|u_\alpha\|\leq C|x|^{-\frac12},
$$
which implies
$$
\lim\limits_{|x|\rightarrow\infty}|u_\alpha(x)|=0\quad\text{uniformly \,for}\,\, \alpha\in(0,\al^*).
$$
Thus, by assumption $(f_1)$, we deduce that there exists $R_2>0$ such that
\begin{equation}\label{eqn:re3}
\aligned
f(u_\al)\leq \frac{3}{4}u_\al,\quad |x|\geq R_2.
\endaligned
\end{equation}
Combining (\ref{eqn:re1})--(\ref{eqn:re3}), let $R=\max\{R_1,R_2\}$ to get for $\al\in(0,\al^*)$
\begin{equation}\label{eqn:re4}
\aligned
    -\triangle {u_\alpha}+ \frac{1}{4}u_\alpha\leq0,\quad |x|\geq R.
\endaligned
\end{equation}
From (\ref{eqn:re4}) and the comparison principle, there exists a constant
$M\geq \frac{C}{R}e^{R/2}$ such that
$$
u_\alpha(x)\leq M \exp\bigg(-\frac{1}{2}|x|\bigg) \quad\text{for}\,\, |x|\geq R.
$$
Here $R,M$ are independent of $\al$.
\end{proof}

\noindent \emph{Proof of Theorem 1.1.} We are now in the position to carry out the proof of our main result which we divide into two steps:

\medskip

\noindent \emph{Step 1.} Let us show that $u_0\in H_r^{1}(\R^2)$ satisfies $I'(u_0)=0$.\

\noindent For any $\varphi\in C_0^\infty(\R^2)$, we have
\begin{equation}\label{eqn:Th1}
\aligned
I'_\alpha(u_\alpha)\varphi
=&\int_{\R^2}\nabla u_\al\nabla\varphi {\rm d}x+\int_{\R^2} u_\alpha\varphi{\rm d}x\\
&-\int_{\R^2}\int_{\R^2}\frac{|x-y|^{-\alpha}-1}{\alpha}F(u_\alpha(y)){\rm d}y f(u_\alpha)
\varphi {\rm d}x.
\endaligned
\end{equation}
Recalling (\ref{eqn:fun3+f}), we have
\begin{equation}\label{eqn:Th1+}
\aligned
\int_{\R^2}f(u_\alpha)\varphi {\rm d}x\rightarrow\int_{\R^2}f(u_0)\varphi {\rm d}x,\quad\text{as}\,\,\al\rightarrow0^+.
\endaligned
\end{equation}
Then it follows from Lemma \ref{Lem:G} that for any fixed $\varphi\in C_0^\infty(\R^2)$, we have for $\al\in(0,\al^*)$
\begin{equation}\label{eqn:Th2-}
\aligned
&\bigg|\frac{|x-y|^{-\alpha}-1}{\alpha}1_{|x-y|\leq1}F(u_\alpha(y))f(u_\alpha(x))\varphi(x)\bigg|\\
&\leq
\bigg|\frac{1}{|x-y|^{\frac{4(\omega-1)}{3\omega}}}1_{|x-y|\leq1}F(u_\alpha(y))f(u_\alpha(x))\varphi(x)\bigg|:=h_\al(x,y),
\endaligned
\end{equation}
which, together with Lemma \ref{Lem:jianjin+} and (\ref{eqn:Th1+}), yields
$$
\aligned
&\bigg|\int_{\R^2}\int_{\R^2}h_\al(x,y)-\frac{1}{|x-y|^{\frac{4(\omega-1)}{3\omega}}}1_{|x-y|\leq1}F(u_0(y))f(u_0(x))\varphi(x){\rm d}x{\rm d}y\bigg|\\
\leq&\bigg|\int_{\R^2}\int_{\R^2}\frac{1}{|x-y|^{\frac{4(\omega-1)}{3\omega}}}1_{|x-y|\leq1}\big(F(u_\al(y))-F(u_0(y))\big)
f(u_\al(x))\varphi(x){\rm d}y{\rm d}x\bigg|\\
&+\bigg|\int_{\R^2}\int_{\R^2}\frac{1}{|x-y|^{\frac{4(\omega-1)}{3\omega}}}1_{|x-y|\leq1}F(u_0(y)){\rm d}y\big( f(u_\al(x))-f(u_0(x))\big)\varphi(x){\rm d}x\bigg|\\
=:& I_1+o_\al(1).
\endaligned
$$
Observe by Lemma \ref{Lem:jianjin+} that there exists $C>0$ independent of $\al$ such that
$$
\aligned
I_1&\leq C\int_{\R^2}\int_{\R^2}1_{|x-y|\leq1}f(u_\al(x))\varphi(x){\rm d}y{\rm d}x\\
&\leq C\int_{\R^2}\int_{\R^2}1_{\Omega}{\rm d}y f(u_0(x))\varphi(x){\rm d}x+o_\al(1),
\endaligned
$$
where $|\Omega|$ is large enough. Thus, the Lebesgue dominated convergence theorem implies $I_1\rightarrow0$, as $\al\rightarrow0^+$.
Furthermore, since $\{h_\al\}$ has a strongly convergent subsequence in $L^1(\R^2)$,
we use the Lebesgue dominated convergence theorem in (\ref{eqn:Th2-}) to get
\begin{equation}\label{eqn:Th2--}
\aligned
&\int\int_{|x-y|\leq1}\frac{|x-y|^{-\alpha}-1}{\alpha}F(u_\alpha(y)){\rm d}y f(u_\alpha(x))\varphi(x) {\rm d}x\\
&\rightarrow
-\int\int_{|x-y|\leq1}\ln(|x-y|)F(u_0(y)){\rm d}yf(u_0(x))\varphi(x) {\rm d}x.
\endaligned
\end{equation}
When $|x-y|\geq1$, there exists $\tau=\tau(|x-y|)\in(0,1)$ such that
\begin{equation}\label{eqn:Th3}
0\geq G_{\alpha}(x-y)=\frac{|x-y|^{-\alpha}-1}{\alpha}=-|x-y|^{-\tau\alpha}\ln|x-y|,
\end{equation}
where $\tau$ depends on $|x-y|$.
Since $\varphi $ has a compact support, from Lemma \ref{Lem:R2} and $(f_1)$ we have
\begin{equation}\label{eqn:Th4}
\aligned
&\bigg|\frac{|x-y|^{-\alpha}-1}{\alpha}\cdot1_{|x-y|\geq1}F(u_\alpha(y))f(u_\alpha(x))\varphi(x)\bigg|\\
&=\bigg||x-y|^{-\tau\alpha}\ln|x-y|\cdot1_{|x-y|\geq1}F(u_\alpha(y))f(u_\alpha(x))\varphi(x)\bigg| \\
&\leq\bigg|(|x|+|y|)F(u_\alpha(y))f(u_\alpha(x))\varphi(x)\bigg|\\
&\leq\bigg|C\bigg(F(u_\alpha(y))+M|y| e^{-\frac{1}{2}|y|}\bigg)f(u_\alpha(x))\varphi(x)\bigg|\\
 &=:\bar{h}_\alpha(x,y).
\endaligned
\end{equation}
Using Lemma \ref{Lem:R2} and Lemma \ref{Lem:compact0}, we have that $\{\bar{h}_\al\}$ has
a strongly convergent subsequence in $L^1(\R^2\times\R^2)$.
Combining (\ref{eqn:Th3}) with (\ref{eqn:Th4}), similarly to \eqref{eqn:Th2-}, by the Lebesgue dominated convergence theorem, one has
\begin{equation}\label{eqn:Th5}
\aligned
&\int\int_{|x-y|\geq1}\frac{|x-y|^{-\alpha}-1}{\alpha}F(u_\alpha(y)){\rm d}y f(u_\alpha(x))\varphi(x) {\rm d}x\\
&\rightarrow
-\int\int_{|x-y|\geq1}\ln|x-y|F(u_0(y)){\rm d}yf(u_0(x))\varphi(x) {\rm d}x.
\endaligned
\end{equation}
Fatou's lemma yields
\begin{equation}\label{eqn:Th6}
\aligned
&\bigg|\int_{\R^2}\int_{\R^2}\ln|x-y|F(u_0(y)){\rm d}y F(u_0(x)){\rm d}x\bigg|\\
&\leq \liminf\limits_{\alpha\rightarrow0}
\bigg(\int\int_{|x-y|\leq1}G_{\alpha}(x-y)F(u_\al(y)){\rm d}y F(u_\al(x)) {\rm d}x\\
&-\int\int_{|x-y|\geq1} G_{\alpha}(x-y)F(u_\al(y)){\rm d}y F(u_\al(x)) {\rm d}x\bigg).
\endaligned
\end{equation}
By the Hardy-Littlewood-Sobolev inequality, (\ref{eqn:al1}) and  Lemma \ref{Lem:G}, there exists $C_\omega>0$ such that
\begin{equation}\label{eqn:Th7}
\aligned
&\int\int_{|x-y|\leq1}G_{\alpha}(x-y)F(u_\al(y)){\rm d}y F(u_\al(x)) {\rm d}x\\
&\leq\int_{\R^2}\int_{\R^2}\frac{F(u_\al(y))}{|x-y|^{\frac{4(\omega-1)}{\omega}}}{\rm d}y F(u_\al(x)) {\rm d}x\\
&\leq C_\omega\bigg(\int_{\R^2}F^\omega(u_\al){\rm d}x\bigg)^2\leq C
\endaligned
\end{equation}
uniformly for $\alpha\in \left(0, \frac{4(\omega-1)}{3\omega}\right)$ sufficiently small, where $\omega$ as in Lemma \ref{Lem:jianjin+}.
By Remark \ref{rem:5.3} and (\ref{eqn:Th7}), we also have
\begin{equation}\label{eqn:Th8}
\aligned
&\int\int_{|x-y|\geq1} G_{\alpha}(x-y)F(u_\al(y)){\rm d}y F(u_\al(x)) {\rm d}x\\
&\leq I_\al(u_\alpha)+\int\int_{|x-y|\leq1} G_{\alpha}(x-y)F(u_\al(y)){\rm d}y F(u_\al(x)) {\rm d}x
-\frac{1}{2}\|u_\alpha\|^{2}\\
&< +\infty
\endaligned
\end{equation}
uniformly for $\alpha$ small enough. Joining (\ref{eqn:Th6}), (\ref{eqn:Th7}) and (\ref{eqn:Th8}) gives the following
\begin{equation}\label{eqn:Th9}
\bigg|\int_{\R^2}\int_{\R^2}\ln|x-y|F(u_0(y)){\rm d}y F(u_0(x)){\rm d}x\bigg|<+\infty.
\end{equation}
Based on (\ref{eqn:Th2--}), (\ref{eqn:Th5}) and (\ref{eqn:Th9}), by taking the limit in (\ref{eqn:Th1}),
we have $I'(u_0)=0$ with $I(u_0)<+\infty$, that is, $u_0\in H_r^{1}(\R^2)$ solves equation (\ref{eqn:log-choquard0}).

\medskip

\noindent \emph{Step 2.}
It remains to show that $u_0\neq0$ and that $u_\alpha\rightarrow u_0$ in $H_r^{1}(\R^2)$. Assume by contradiction that $u_\alpha\rightharpoonup 0$ in $H_r^{1}(\R^2)$,
as well as $u_\alpha\rightarrow0$ in $L^t(\R^2)$ for $t\in(2,+\infty)$. Similarly to (\ref{eqn:con10}), we obtain
$\int_{\R^2}f(u_\alpha)u_\alpha {\rm d}x=o_\al(1)$.
So, by Lemma \ref{Lem:G}, Lemma \ref{Lem:jianjin+}, we have
$$
\aligned
I'_\alpha(u_\alpha)u_\alpha&=\|u_\alpha\|^{2}
-\int_{\R^2}\int_{\R^2}G_{\alpha}(x-y)F(u_\alpha(y))f(u_\alpha(x))u_\alpha(x){\rm d}x{\rm d}y\\
&\geq \|u_\alpha\|^{2}-\int\int_{|x-y|\leq1}G_{\alpha}(x-y)F(u_\alpha(y))f(u_\alpha(x))u_\alpha(x){\rm d}x{\rm d}y\\
&\geq \|u_\alpha\|^{2}-\int\int_{|x-y|\leq1}\frac{1}{|x-y|^{\frac{4(\omega-1)}{\omega}}}F(u_\alpha(y))f(u_\alpha(x))u_\alpha(x){\rm d}x{\rm d}y\\
&\geq \|u_\alpha\|^{2}-C\int_{\R^2}f(u_\alpha(x))u_\alpha(x){\rm d}x,
\endaligned
$$
and thus $u_\alpha\rightarrow0$ in $H_r^{1}(\R^2)$, as $\al\rightarrow0^+$. Then, according to Remark \ref{rem:5.3}, (\ref{eqn:Th4}), Lemma \ref{Lem:R2} and Lemma \ref{Lem:compact0}, we have
$$
\aligned
a&\leq I_\al(u_\alpha)\\
&=\frac{1}{2}\|u_\alpha\|^2-\frac{1}{2}\int_{\R^2}
\int_{\R^2}G_{\alpha}(x-y)F(u_\al(y)) F(u_\al(x)){\rm d}x{\rm d}y\\
&\leq-\frac{1}{2}\int\int_{|x-y|\geq1}G_{\alpha}(x-y)F(u_\al(y)) F(u_\al(x)){\rm d}x{\rm d}y+o_\al(1)\\
&\leq C\int_{\R^2}|x|F(u_\al(x)){\rm d}x\int_{\R^2}F(u_\al(y)){\rm d}y+o_\al(1)\\
&\leq C\bigg(\int_{|x|\leq R}|x|F(u_\al(x)){\rm d}x+\int_{|x|\geq R}|x|F(u_\al(x)){\rm d}x\bigg)+o_\al(1)\\
&\leq C\bigg(\int_{|x|\leq R}F(u_\al(x)){\rm d}x+\int_{|x|\geq R}\frac{C|x|}{Me^{\frac{|x|}{2}}}{\rm d}x\bigg)^2+o_\al(1)\\
&=o_\al(1),
\endaligned
$$
which yields a contradiction. So, $u_0\not=0$.
Furthermore, similarly to (\ref{eqn:Th2-}), (\ref{eqn:Th5}),
by Lemma \ref{Lem:R2} and the Lebesgue dominated convergence theorem, we have
\begin{equation}\label{eqn:Th10}
\aligned
&\int_{\R^2}\int_{\R^2}\frac{|x-y|^{-\alpha}-1}{\alpha}F(u_\al(y))f(u_\al(x))u_\al(x) {\rm d}y{\rm d}x\\
&\rightarrow
-\int_{\R^2}\int_{\R^2}\ln|x-y|F(u_0(y))f(u_0(x))u_0(x) {\rm d}x{\rm d}y,
\endaligned
\end{equation}
from which we conclude that $u_\alpha\rightarrow u_0$ in $H_r^1(\R^2)$, as $\alpha\rightarrow0^+$.
\qed

\end{document}